\documentclass{amsart}

\usepackage{amssymb,amscd}

\def\Cur{\mathrm{Cur}}
\def\Perm{\mathrm{Perm}}

\newtheorem{theorem}{Theorem}
\newtheorem{lemma}{Lemma}
\newtheorem{corollary}{Corollary}

\theoremstyle{definition}
\newtheorem{definition}{Definition}
\newtheorem{example}{Example}
\newtheorem{problem}{Problem}

\title{On special identities for dialgebras}

\author{P.~S.~Kolesnikov and V.~Yu.~Voronin}

\address{Sobolev Institute of Mathematics,
630090 Novosibirsk, Russia; {\tt pavelsk@math.nsc.ru};\\
Novosibirsk State University, 630090 Novosibirsk, Russia}

\begin{document}

\begin{abstract}
For every variety of algebras over a field, there is a natural definition of
a corresponding variety of dialgebras (Loday-type algebras).
In particular, Lie dialgebras are equivalent to Leibniz algebras.
We use an approach based on the notion of an operad to study 
the problem of finding special identities for dialgebras. It is proved that 
all polylinear special identities for dialgebras can be 
obtained from special identities for corresponding algebras
by means of a simple procedure. A particular case of this 
result confirms the conjecture by 
M. Bremner, R. Felipe, and J. Sanchez-Ortega,
arXiv:1108.0586.
\end{abstract}

\maketitle

\section{Introduction}
The notion of a Leibniz algebra appeared first in \cite{Blok} and
later independently in \cite{Lod93} gave rise to a series of
research devoted to the theory of dialgebras. By definition, a
(left) Leibniz algebra
is a linear space with a bilinear operation $[\cdot, \cdot ]$ which
satisfies the Jacobi identity in the form
$[x,[y,z]] = [[x,y],z] + [y,[x,z]]$, i.e.,  the operator of left
multiplication $[x,\cdot ]$ is a derivation. This is one of the
most studied noncommutative analogues of Lie algebras.

Various classes of dialgebras
appeared in the literature since they are related to Leibniz algebras in the same way
as the corresponding classes of ordinary algebras are related to Lie algebras.
Associative dialgebras were introduced in \cite{LP93}
as analogues of associative enveloping algebras
for Leibniz algebras, alternative dialgebras appeared in \cite{Liu05} in the study of universal central
extensions for Leibniz algebras, Jordan dialgebras
(first under the name of quasi-Jordan algebras)
were proposed in \cite{VF}, see also \cite{Br09} and \cite{Kol2008}.
All dialgebras of these classes are linear spaces equipped by two bilinear operations
$\vdash $ and $\dashv $ such that
\begin{equation}\label{eq:0-ident(2)}
(x\dashv y)\vdash z = (x\vdash y)\vdash z , \quad x\dashv (y\vdash z)=x\dashv (y\dashv z).
\end{equation}
These identities are common for associative, alternative, Jordan dialgebras mentioned above,
and they also hold for Leibniz algebras provided that $a\vdash b = [a,b]$,
$a\dashv b = -[b,a]$.
Other defining identities of these varieties initially appeared from
a posteriori considerations
motivated by relations with Leibniz algebras.
For example, a dialgebra is associative if in addition to
\eqref{eq:0-ident(2)} the following identities hold
\[
\begin{gathered}
x\vdash (y\vdash z) =  (x\vdash y)\vdash z, \quad
x\vdash (y\dashv z) =  (x\vdash y)\dashv z, \\
x\dashv (y\dashv z) =  (x\dashv y)\dashv z.
\end{gathered}
\]
Then the same space with respect to the new operation
$[a,b] = a\vdash b - b\dashv a$ is a Leibniz algebra.
A systematical study of this relation between Leibniz algebras and
associative dialgebras may be found in \cite{Lod2001}.

An idea of a more conceptual approach to the definition of what should be
called a dialgebra associated with a
given variety $\mathfrak P$ of ordinary algebras was proposed in
\cite{Chap2001} in the case of associative algebras:
It was shown that the operad governing the variety of associative dialgebras in the
sense of  \cite{LP93}  coincides with the Hadamard product $\mathrm{As}\otimes \Perm $,
where
$\mathrm{As}$ is the operad governing the variety of associative algebras and
$\Perm $ is the operad governing associative algebras
satisfying the left commutativity
relation
$(xy)z - (yx)z =0 $.

For an arbitrary variety $\mathfrak P $ of ordinary algebras with one binary
operation (governed by an operad $\mathcal P$), the algorithm
proposed in \cite{Kol2008} and \cite{Pozh09} allows to deduce the
defining identities for the class of (di-)algebras governed by
the operad $\mathcal P\otimes \Perm$ starting with the defining identities of $\mathfrak P $.
In \cite{BSO_Triple}, this algorithm was generalized to the case of arbitrary varieties
of algebras
of any type (i.e., linear spaces with a family of polylinear operations
of arbitrary arity).
In this note, we will show that this generalized algorithm also leads to the class
of $\mathcal P\otimes \Perm$-algebras.
This is why we denote by $\mathcal P\otimes \Perm$  by
$\mathrm{di}\mathcal P$.

This fact allows to consider
a series of questions devoted to elementary properties and relations between
various classes of dialgebras from a unified point of view.
In particular, a morphism of operads
$\omega: \mathcal P\to \mathcal R$  always gives rise to
a functor from the category of $\mathcal R$-algebras to the category of
$\mathcal P$-algebras. So are the well-known functors:

\begin{tabular}{c|c|c}
 $\mathcal R$ & $\mathcal P$ &  $\omega $ \\
\hline
Associative & Lie & $x_1x_2\mapsto x_1x_2-x_2x_1 $ \\
\hline
Associative & Jordan  & $x_1x_2\mapsto x_1x_2+x_2x_1 $ \\
\hline
Alternative & Jordan & $x_1x_2\mapsto x_1x_2+x_2x_1 $ \\
\hline
Alternative & Mal'cev & $x_1x_2\mapsto x_1x_2-x_2x_1 $ \\
\hline
Associative & Jordan triple  & $\langle x_1,x_2,x_3\rangle \mapsto x_1x_2x_3+x_3x_2x_1 $ \\
 & system & \\
\hline
Jordan & Jordan triple & $\langle x_1,x_2,x_3\rangle \mapsto (x_1x_2)x_3 - (x_1x_3)x_2$ \\
& system & ${}+ x_1(x_2x_3) $  \\
\hline
\end{tabular}

\medskip
For each triple $(\mathcal P, \mathcal R, \omega )$ as above
the following {\em speciality problem\/} makes sense:
Whether the variety generated by all those $\mathcal P$-algebras obtained
from $\mathcal R$-algebras coincides with the class of all $\mathcal P$-algebras?
If no, what are the identities separating these classes (special identities)?
The same question is actual for dialgebras:
The corresponding varieties of di$\mathcal R$- and di$\mathcal P$-algebras
are related by a functor raising from the morphism
$\omega \otimes \mathrm{id}$ of operads.

The purpose of this note is to show that the speciality problem
for dialgebras raising from
the triple
$(\mathrm{di}\mathcal P, \mathrm{di}\mathcal R, \omega \otimes \mathrm{id})$
can always be solved modulo the same problem for
ordinary algebras.

\section{The BSO algorithm}
Let us start with the construction from \cite{BSO_Triple}, assuming
the base field $\Bbbk $ is of zero characteristic.

Let $A$ be an associative algebra over $ \Bbbk $ equipped by new $n$-ary operation
\begin{equation}\label{eq:omegadefin}
\omega(x_1,\ldots,x_n)=\sum_{\sigma\in S_n} \alpha_\sigma x_{\sigma(1)} \dots x_{\sigma(n)},
\end{equation}
where $\alpha_\sigma\in\Bbbk $.

Choose an index $i \in \{1, \dots, n\}$. One may rewrite \eqref{eq:omegadefin} as
\[
 \omega(x_1,\ldots,x_n)=\sum_{j=1}^n
 \sum_{S_n^{ji}} \alpha_\sigma x_{\sigma(1)} \dots x_{\sigma(j-1)} x_i x_{\sigma(j+1)}  \dots x_{\sigma(n)},
\]
where $S_n^{ji}$~--- is a set of permutations such that $\sigma(j)=i$.

Denote by $\mathrm{Id}(\omega )$ the set of all polylinear identities satisfied for all
$n$-ary algebras obtained in this way from associative ones.

Starting from the identities $\mathrm{Id}(\omega )$, one may canonically construct
a set of identities
$\mathrm{Id}(\omega )^{(2)}$
of type $\{\omega_1, \dots, \omega _n\}$,
where each $\omega _i $ is an $n$-ary operation. The algorithm of such a construction
was described in \cite{BSO_Triple} (as a KP algorithm), and also
in Section~\ref{subsec:Dialg}.

On the other hand, consider the following operations on an associative dialgebra $D$:
\[
 \omega_i(x_1,\dots,x_n)=\sum_{j=1}^n \sum_{S_n^{ji}}
 \alpha_\sigma x_{\sigma(1)} \vdash \dots \vdash x_{\sigma(j-1)}\vdash x_i \dashv x_{\sigma(j+1)} \dashv \dots\dashv  x_{\sigma(n)},
\]
$i=1,\dots, n$ (the bracketing is not essential here).
The family of $n$-ary operations $\omega_1,\dots,\omega_n$
obtained are denoted by $\mathrm{BSO}\,(\omega)$.
Let $\mathrm{Id}(\mathrm{BSO}\,(\omega ))$ stand for the set of all polylinear identities satisfied for all
algebras with $n$-ary operations $\mathrm{BSO}\,(\omega )$ obtained in this way from associative dialgebras.

\begin{problem}[\cite{BSO_Triple}]\label{probBSO}
Let $\mathrm{char}\,\Bbbk =0$. Prove that for every choice of
 $\omega$ we have $\mathrm{Id}(\mathrm{BSO}\,(\omega))=\mathrm{Id}(\omega)^{(2)}$.
\end{problem}

Next, suppose $\mathrm{char}\,\Bbbk =p>0$. Then the relation 
in Problem \ref{probBSO} is not valid in general, but it is reasonable 
to state

\begin{problem}[\cite{BSO_Triple}]\label{probBSOp}
For 
$\mathrm{char}\,\Bbbk =p>0$ and $d<p$, prove that for every choice of
 $\omega$ we have $\mathrm{Id}_d(\mathrm{BSO}\,(\omega))=\mathrm{Id}_d(\omega)^{(2)}$,
where $\mathrm{Id}_d(\cdot )$ stands for the subset 
of identities of degree $d$ in $\mathrm{Id}(\cdot)$.
\end{problem}

In this paper, we solve these problems.

\section{Preliminaries in operads}
In this section, we state the necessary notions
of the operad theory following mainly \cite{GinzKapr1994},
with a particular accent on the operads governing
varieties of algebras.

A {\em language\/} $\Omega $ is a set of functional
symbols $\{f_i\mid i\in I\}$ equipped by an {\em arity function\/}
$\nu: f_i\mapsto n_i\equiv \nu(f_i)\in \mathbb N$.
An $\Omega $-{\em algebra\/} is a linear space $A$ over a base field
$\Bbbk$
endowed with linear maps
$f_i^A: A^{\otimes n_i}\to A$, $i\in I$ \cite{Kur60}.
Below, we will use the term {\em algebra of type $\Omega $\/}
for an $\Omega $-algebra
to avoid confusion. In this paper, we assume $n_i\ge 2$.

Denote by $\mathcal F_\Omega \langle X\rangle $ the
free algebra of type $\Omega $
generated by the countable set $X=\{x_1,x_2, \dots \}$. The linear
basis of this algebra consists of all terms of type $\Omega $ in variables
from $X$. Let us call such terms {\em monomials}, their linear
combinations (elements of the free algebra) are called {\em
polynomials}.

For every $n\in \mathbb N$ consider the space
$\mathcal F_\Omega(n)$ of all polylinear polynomials of degree
$n$ in $x_1,\dots, x_n$.
The composition
\[
\gamma_{m_1,\dots, m_n}: \mathcal F_\Omega(n)\otimes
  \mathcal F_\Omega(m_1)\otimes \dots \otimes  \mathcal F_\Omega(m_n)
\to  \mathcal F_\Omega(m_1+\dots + m_n)
\]
of such maps is naturally defined by the rule
\[
\gamma_{m_1,\dots, m_n}(f; g_1,\dots, g_n)
= f(g_1(x_1,\dots, x_{m_1}), g_2(x_{m_1+1}, \dots , x_{m_1+m_2}),
 \dots ),
\]
where
$f(x_1,\dots, x_n)\in \mathcal F_\Omega (n)$,
$g_i(x_1,\dots, x_{m_i})\in \mathcal F_\Omega (m_i)$, $i=1,\dots, n$; the result belongs to
$\mathcal F_\Omega (m_1+\dots +m_n)$.
The simplest term $x_1\in \mathcal F_\Omega (1)$ behaves as an identity with respect to
this  composition.
Symmetric groups $S_n$ act on
$\mathcal F_\Omega (n)$ by permutations of variables.

The collection of spaces $\{\mathcal F_\Omega(n)\}_{n\in \mathbb N}$
together
with above-mentioned composition rule, identity element, and
$S_n$-action  is
a particular case of an operad
which is natural to call
the {\em free operad\/} $\mathcal F_\Omega $ generated by~$\Omega $.

Given two operads $\mathcal P$ and $\mathcal R$,
a morphism $\alpha : \mathcal P\to \mathcal R$
is just a family of $S_n$-linear maps
$\{\alpha(n)\}_{n\in \mathbb N}$,
\[
\alpha(n): \mathcal P(n)\to \mathcal R(n)
\]
preserving the composition and the identity
element.
The kernel of $\alpha $ is the collection of subspaces (even $S_n$-submodules)
$\mathrm{Ker}\,\alpha(n)\subseteq \mathcal P(n)$, $n\in \mathbb N$,
which is closed with respect to compositions in the obvious sense.
Such a family of subspaces is called an {\em operad ideal\/} in
$\mathcal P$.

To define a morphism $\pi $ from $ \mathcal F_\Omega$ to an operad $\mathcal P$
it is enough to determine $\pi(n_i)(f_i)$, $f_i = f_i(x_1,\dots, x_{n_i})\in \mathcal F_{\Omega}(n_i)$,
where $f_i$ range through the language $\Omega $, $n_i=\nu(f_i)$.
Moreover, every family $g_i\in \mathcal P(n_i)$, $i\in I$, defines
a unique morphism of operads $\pi:\mathcal F_\Omega \to \mathcal P$
such that $\pi(n_i)(f_i)=g_i$.

Every linear space $A$ gives rise to an operad $\mathcal E(A)$,
the operad of endomorphisms of $A$.
Namely,
$\mathcal E(A)(n)= \mathrm{Hom}\,(A^{\otimes n}, A)$, $n\in \mathbb N$,
compositions and $S_n$-actions are defined in the ordinary way.

A structure of an algebra of type $\Omega $ on a linear
space $A$ may be identified with a
morphism of operads $\alpha : \mathcal F_\Omega \to \mathcal E(A)$
such that $\alpha(n_i)(f_i) = f_i^A$, $i\in I$.
Conversely, every morphism $\alpha : \mathcal F_\Omega \to \mathcal E(A)$
defines a structure of
an algebra on~$A$.

Suppose $\mathfrak P$ is a variety of algebras of type $\Omega $
defined by polylinear identities (this is a generic case if
$\mathrm{char}\,\Bbbk =0$). Then the following consideration
makes sense.

Let $T(\mathfrak P)$ be the ideal of identities (T-ideal)
 in $\mathcal F_\Omega \langle X\rangle $
corresponding to the variety
$\mathfrak P$.
Denote
$\mathcal P(n) = \mathcal F_\Omega (n)/ (T(\mathfrak P )\cap \mathcal F_\Omega (n))$,
$n\in \mathbb N$.
The composition rule and $S_n$-actions are well-defined on the family
$\{\mathcal P(n)\}_{n\in \mathbb N}$, so this collection
is also an operad. Such an operad
is said to be
the {\em governing operad\/} for the variety
$\mathfrak P$. There exists a natural quotient morphism
$\pi: \mathcal F_\Omega \to \mathcal P$.
If $S$
is a defining family of polylinear identities of the variety
$\mathfrak P $ then the kernel of $\pi $ is exactly the operad ideal
generated by $S$ in $\mathcal F_\Omega $.

Every algebra $A$ from the variety
$\mathfrak P $ is determined by a composition
$\pi\circ \bar \alpha$, where $\bar \alpha$
is a morphism from $\mathcal P$ to $\mathcal E(A)$.
Thus,  $A$ is defined
by a morphism of operads $\mathcal P\to \mathcal E(A)$.
Conversely, every morphism of this kind defines an algebra structure
on $A$, and the algebra obtained belongs to $\mathfrak P$.

In general, given an operad $\mathcal P$,
a {\em $\mathcal P$-algebra\/} is a pair
$(A,\alpha)$ of a linear space $A$ and a morphism of operads
 $\alpha:\mathcal P \to \mathcal E(A)$.

\section{Conformal algebras}
The notion of a conformal algebra was introduced in \cite{Kac1996} as
tool of vertex operator algebra study. In a more general context,
a conformal algebra is a pseudo-algebra over the polynomial algebra
$\Bbbk[T]$ in one variable \cite{BDK2001}. Here we consider the last
approach for arbitrary set of operations~$\Omega $.

As we have already mentioned, every linear space $A$ gives rise to
the operad $\mathcal E(A)$. A similar construction exists
for left unital modules over a cocommutative bialgebra $H$. Suppose
$M$ is such a module, then denote
\[
\mathcal E^*(M)(n)=\mathrm{Hom}_{H^{\otimes n}} (M^{\otimes n}, H^{\otimes n}\otimes_H M).
\]
Hereinafter, the symbol $\otimes $ without a subscript stands for the tensor product
of spaces over the base field. The space $H^{\otimes n}$ is considered as the
outer product of regular right $H$-modules, i.e.,
\[
(h_1\otimes \dots \otimes h_n)\cdot h = \sum\limits_{(h)} h_1h_{(1)}\otimes \dots \otimes h_nh_{(n)},
\]
where 
$\sum\limits_{(h)} h_{(1)}\otimes \dots \otimes h_{(n)}$ 
is the value of $n$-iterated coproduct on~$h$.
Compositions $\gamma_{m_1,\dots, m_n}$ of such maps
and the action of $S_n$ on $\mathcal E^*(M)(n)$ were defined in \cite{BDK2001},
see also \cite{Kol2008}
(one needs cocommutativity of $H$ to ensure the action of $S_n$ is well-defined).

A {\em conformal algebra\/} over $H$ is a pair $(C, \alpha)$, where $C$ is an $H$-module as above and
$\alpha : \mathcal F_\Omega \to \mathcal E^*(C)$ is a morphism of operads. If $\alpha $ splits into
$\mathcal F_\Omega \overset{\pi}\to \mathcal P \overset{\bar\alpha}\to \mathcal E^*(C) $
then $C$ is said to be a $\mathcal P$-conformal algebra.

A simple but important example of a conformal algebra may be constructed as follows.
Let $(A,\alpha )$ be a $\mathcal P$-algebra. Consider the free $H$-module
$C=H\otimes A$ and define $\beta=\Cur\,\alpha
 : \mathcal P\to \mathcal E^*(C)$ by the rule
\[
\beta(n) (f): (h_1\otimes a_1)\otimes \dots \otimes (h_n\otimes a_n)\mapsto
  (h_1\otimes \dots \otimes h_n)\otimes_H \alpha(f)(a_1\otimes \dots \otimes a_n),
\]
$f\in \mathcal P(n)$, $h_k\in H$, $a_k\in A$. This is a morphism of operads, and
the $\mathcal P$-conformal algebra $(C,\beta )$ obtained is denoted
$(\Cur\,A, \Cur\,\alpha )$,
the {\em current\/} conformal algebra over~$A$.

The correspondence $A \mapsto \Cur\,A$ is a functor from the category of
$\mathcal P$-algebras to the category of $\mathcal P$-conformal
algebras: Every
morphism $\varphi $ between $\mathcal P$-algebras
can be continued by $H$-linearity to the morphism $\Cur\,\varphi$ of
the corresponding current algebras.

\section{The operad Perm}
The operad $\Perm$ introduced in \cite{Chap2001} is given
by a family of spaces $\Perm(n)=\Bbbk^n$ with
natural composition rule
\[
\gamma_{m_1,\dots , m_n}: e_k^{(n)}\otimes
 e_{j_1}^{(m_1)}\otimes \dots \otimes e_{j_n}^{(m_n)}
=
e^{(m_1+\dots + m_n)}_{m_1+\dots + m_{k-1}+j_k},
\]
where $e_k^{(n)}$, $k=1,\dots, n$, is the standard basis
of $\Bbbk^n$, $n\in \mathbb N$.
Symmetric groups $S_n$ act on $\Perm(n)$ by permutations
of coordinates.

Let $\mathcal P$ be an operad.
Denote by
$\mathrm{di}\mathcal P$ the Hadamard product
$\mathcal P\otimes \Perm$:
$\mathrm{di}\mathcal P (n)  =\mathcal P(n) \otimes \Perm (n)$,
compositions and $S_n$-action are defined in the
componentwise way.

Let us fix a cocommutative bialgebra $H$,
and let $\varepsilon $ stand for its counit.
A left unital $H$-module $C$ is in particular a linear space over the
base field $\Bbbk $.
For every $n\in \mathbb N$ consider $\Bbbk $-linear maps $\mu^k_n $,
$k=1,\dots, n$, from
$H^{\otimes n}\otimes _H C$ to $C$ defined by
\[
\mu_n^k: (h_1\otimes \dots \otimes h_n)\otimes_H c
\mapsto \varepsilon(h_1\,\overset{k}{\hat\dots} \,h_n)h_k c.
\]

\begin{lemma}\label{lem:ZeroMorphism}
If $(C, \alpha )$ is a $\mathcal P$-conformal algebra then
the family of maps
$\{\alpha^{(0)}(n)\}_{n\in \mathbb N}$,
$\alpha^{(0)}(n) : \mathrm{di} \mathcal P(n) \to \mathcal E(C)(n)$,
defined by
\begin{equation}\label{eq:(0)functor}
\alpha^{(0)}(n)(f\otimes e_k^{(n)})=
 \alpha(n)(f)\circ \mu_n^k ,
\end{equation}
$f\in \mathcal P(n)$, $k=1,\dots, n$, $a_j\in C$,
defines a morphism $\alpha^{(0)}$ of operads.
\end{lemma}

\begin{proof}
First, note that $\alpha^{(0)}(n)$ is $S_n$-linear.
Indeed,
\[
\alpha^{(0)}(n): (f\otimes e_k^{(n)})^\sigma =
f^\sigma \otimes e_{\sigma(k)}^{(n)}
\mapsto
\alpha(n)(f)^\sigma \circ \mu_n^{\sigma(k)}
= (\alpha(n)(f)\circ \mu_n^k)^\sigma
\]
since the action of $\sigma $ on $\mathcal E^*(C)$
permutes the arguments of $\alpha(n)(f)$
together with tensor factors in $H^{\otimes n}\otimes _H C$,
see \cite{Kol2008}.

Next, this is obvious that $\alpha^{(0)}(1)$ preserves the identity.

Finally, consider a composition
$\gamma_{m_1,\dots, m_n}(f; g_1,\dots, g_n)$ in $\mathcal P$.
By abuse of notations, assume
\[
\alpha(m_l)(g_l): a_1^{(l)}\otimes \dots \otimes a_{m_l}^{(l)}
\mapsto F^{(l)}\otimes _H b^{(l)}, \quad F^{(l)}\in H^{\otimes m_l},
\]
$l=1,\dots, n$,
and
\[
\alpha(n)(f): b^{(1)}\otimes  \dots \otimes b^{(n)}\mapsto G\otimes _H c,
\quad G\in H^{\otimes n},
\]
$a_j^{(l)}, b^{(l)}, c \in C$.
Then by definition
\begin{multline}\label{eq:HmodComp}
\alpha(m_1+\dots+m_n)(\gamma_{m_1,\dots, m_n}(f; g_1,\dots, g_n)): \\
a_1^{(1)}\otimes \dots \otimes a_{m_1}^{(1)}
 \otimes \dots \otimes
a_1^{(n)}\otimes \dots \otimes a_{m_n}^{(n)}   \\
\mapsto
((F^{(1)}\otimes \dots \otimes F^{(n)})\otimes _H 1)
 (\Delta^{[m_1]}\otimes \dots \otimes \Delta^{[m_n]})G\otimes_H c,
\end{multline}
where $\Delta^{[m]}(h) = \sum\limits_{(h)}h_{(1)}\otimes \dots \otimes h_{(n)}$,
$h\in H$.

Let us fix some
$k\in \{1,\dots, n\}$, $j_l\in \{1,\dots, m_l\}$, $l=1,\dots, n$.
Suppose
$F^{(l)}=h_1^{(l)}\otimes \dots \otimes h_{m_l}^{(l)}$,
$G=h_1\otimes \dots \otimes h_n$.
Then by the properties of the counit $\varepsilon$
 the image of the right-hand side of
\eqref{eq:HmodComp} under $\mu_{m_1+\dots+m_n}^{m_1+\dots + m_{k-1}+j_k}$
is equal to
\begin{equation}\label{eq:Comp_of1}
\varepsilon(F^{(1)})\,\overset{k}{\hat\dots}\, \varepsilon(F^{(n)})
\varepsilon(F_{j_k}^{(k)})\varepsilon(G_k) h^{(k)}_{j_k} h_k c,
\end{equation}
where $G_k = h_1\,\overset{k}{\hat\dots}\, h_n $
and $F_{j_k}^{(k)}$ is defined similarly.

On the other hand, let us compute the composition
\[
\gamma_{m_1,\dots, m_n}(\alpha^{(0)}(n)(f\otimes e_k^{(n)});
\alpha^{(0)}(m_1)(g_1\otimes e_{j_1}^{(m_1)}),
\dots,
\alpha^{(0)}(m_n)(g_n\otimes e_{j_n}^{(m_n)}))
\]
in $\mathcal E(C)$.
By \eqref{eq:(0)functor}, we have
\[
\alpha^{(0)}(m_l)(g_l\otimes e_{j_l}^{(m_l)}):
a_1^{(l)}\otimes \dots \otimes a_{m_l}^{(l)}
\mapsto \varepsilon(F_{j_l}^{(l)})h^{(l)}_{j_l} b^{(l)}.
\]
The $H^{\otimes n}$-linearity of $\alpha(n)(f)$ implies
\begin{multline}\label{eq:Comp_of2}
\alpha(n)(f):
\varepsilon(F_{j_1}^{(1)})h^{(1)}_{j_1} b^{(1)}
\otimes \dots \otimes
\varepsilon(F_{j_n}^{(n)})h^{(n)}_{j_n} b^{(n)}   \\
\mapsto
\varepsilon(F_{j_1}^{(1)})\dots \varepsilon(F_{j_n}^{(n)})
(h^{(1)}_{j_1}\otimes \dots \otimes h^{(n)}_{j_n})G\otimes _H c.
\end{multline}
This is now obvious that the image of the right-hand side of
\eqref{eq:Comp_of2} under $\mu_n^k$ coincides with
\eqref{eq:Comp_of1}.
\end{proof}

Hence, every $\mathcal P$-conformal algebra $(C,\alpha)$ gives rise to a
$\mathrm{di}\mathcal P$-algebra $C^{(0)}=(C, \alpha^{(0)})$.
The correspondence $C\mapsto C^{(0)}$ is obviously a functor
from the category of $\mathcal P$-conformal algebras to the
category of $\mathrm{di}\mathcal P$-algebras.

\section{Dialgebras}\label{subsec:Dialg}
Suppose $\mathcal P$ is a quotient operad of $\mathcal F_\Omega$,
$\pi $ is the corresponding morphism.
This is easy to see that
$\mathrm{di}\mathcal P$
is a quotient of
$\mathrm{di}\mathcal F_\Omega =\mathcal F_\Omega \otimes \Perm $
with respect to the morphism
$\pi\otimes \mathrm{id}$.
In general, this is hard to determine the generators and
defining relations of the Hadamard product of two operads,
but due to the nice properties of $\Perm $
this is easy to do for $\mathrm{di}\mathcal P$.

Consider the free operad $\mathcal F_{\Omega^{(2)}}$,
where the language $\Omega^{(2)}$
is constructed in the following way.
If $\Omega =\{f_i \mid i\in I\}$,
$n_i = \nu(f_i)$, then
$\Omega ^{(2)}=\{f_i^{k} \mid i\in I,\, k=1,\dots, n_i\}$,
$\nu(f_i^{k})=n_i$.

Define the morphism
$\zeta_\Omega: \mathcal F_{\Omega^{(2)}}\to\mathrm{di}\mathcal F_\Omega $
in the following way: $\zeta_\Omega (n)$
maps
$f_i^k(x_1,\dots, x_{n_i})\in \mathcal F_{\Omega^{(2)}}(n_i)$
to $f_i(x_1,\dots, x_{n_i})\otimes e_k^{(n_i)}$.
The composition of $\zeta_\Omega $ with
$\pi\otimes \mathrm{id}$ provides
a morphism
$\pi^{(2)}:\mathcal F_{\Omega^{(2)}}\to \mathrm{di}\mathcal P$.

\begin{lemma}\label{lem:SurjPerm}
For every $n\in \mathbb N$ the linear maps
$\zeta_\Omega(n)$ and
$\pi^{(2)}(n)$ are surjective.
\end{lemma}

\begin{proof}
To prove the surjectivity of $\zeta_\Omega $ (and hence of $\pi^{(2)}$)
it is enough to
show that $\mathrm{di}\mathcal F_\Omega$ is generated
by $f_i\otimes e_k^{(n_i)}$, $i\in I$, $k=1,\dots, n_i$.
In the binary case it was actually done in \cite{Val2008}
and \cite{Kol2008}, the general case can be processed
analogously. It is enough to construct a section
$\rho(n): (\mathcal F_\Omega \otimes \Perm)(n) \to
\mathcal F_{\Omega^{(2)}}(n)$ such that
$\rho(n)\circ \zeta_\Omega (n) = \mathrm{id}$.
It was done in \cite{BSO_Triple}, let us
recall here the construction in terms of planar trees.
Every monomial $f\in \mathcal F_{\Omega }(n)$
can be identified with a planar tree with $n$ leaves (variables)
labeled by numbers $1,\dots ,n $ and vertices labeled by
symbols from $\Omega $, the degree (number of out-coming branches) of a vertex labeled by $f_i\in \Omega $
is equal to~$n_i$. Then $f\otimes e^{(n)}_k$ may
be considered as a tree with $k$th emphasized vertex.
To get $\rho(n)(f\otimes e^{(n)}_k)$ we should add superscripts
to the labels of vertices in the tree corresponding to $f$
in the following way. If a $k$th (counting from the left-hand
side) out-coming branch of a vertex labeled by $f_i\in \Omega $ contains the
emphasized leaf then the label is replaced with $f_i^k$.
If neither of the out-coming branches in this vertex contain the
emphasized leaf then the label is replaced with $f_i^1$.
\end{proof}

Suppose $S$ is a set of polylinear
polynomials such that the kernel of $\pi $ is generated by $S$
(e.g., if $\mathcal P$ is a governing operad for a variety
$\mathfrak P $ then $S$ consists of its defining identities).
Consider the operad ideal $J(S)$ in $\mathcal F_{\Omega^{(2)}} $
generated by
\begin{multline}\label{eq:0ident_general}
f^k(x_1, \dots, x_{j-1}, g^l(x_j, \dots, x_{j+m-1}), x_{j+m}, \dots , x_{n+m-1})\\
 - f^k(x_1, \dots, x_{j-1}, g^p(x_j, \dots, x_{j+m-1}), x_{j+m}, \dots , x_{n+m-1}),
\end{multline}
\[
f,g\in \Omega ,\ n=\nu(f),\ m=\nu(g),\
k,j=1,\dots, n,\ k\ne j,\ l,p=1,\dots, m,
\]
and
\begin{equation}\label{eq:DottedIds}
s^k(x_1,\dots, x_n), \quad s\in S\cap \mathcal F_\Omega(n), \ n\in \mathbb N, \ k=1,\dots, n,
\end{equation}
where
$s^k=\rho(n)(s\otimes e_k^{(n)})$.
Denote by $\mathcal P^{(2)}$ the quotient operad of
$\mathcal F_{\Omega^{(2)}}$ with respect to $J(S)$, and let
$\hat \pi^{(2)}$ be the corresponding morphism
from $\mathcal F_{\Omega^{(2)}}$ to $\mathcal P^{(2)}$.

This is easy to see that $J(S)$ is contained in the kernel
of $\pi^{(2)}$. Thus we have the following commutative diagram:
\[
\begin{CD}
\mathcal F_{\Omega^{(2)}} @>\zeta_\Omega >> \mathcal F_\Omega \otimes \Perm @>\mathrm{pr}>> \mathcal F_\Omega \\
@V\hat\pi^{(2)}VV @VV\pi\otimes\mathrm{id} V @VV\pi V \\
\mathcal P^{(2)} @>>> \mathcal P\otimes \Perm @>\mathrm{pr}>> \mathcal P
\end{CD}
\]
where $\mathrm{pr}$ stands for the natural projection.

Our aim is to show that the kernels of $\hat \pi^{(2)}$ and $\pi^{(2)}$
are equal.

Suppose $(A,\alpha)$ is a $\mathcal P^{(2)}$-algebra.
By abuse of notations, let us identify
$f^k\in \mathcal F_{\Omega^{(2)}}(n)$
and their images under $\hat\pi^{(2)}(n)$ in $\mathcal P^{(2)}(n)$.

Let $A_0$ be the $\Bbbk $-linear span of all
\[
\alpha(n)(f_i^p - f_i^l)(a_1,\dots, a_{n_i}),
 \quad i\in I,\ a_j\in A,\ p,l=1,\dots,n_i.
\]
It follows from \eqref{eq:0ident_general} that $A_0$
is an ideal in the algebra $A$.
Indeed, for every $f\in \Omega $
\[
\alpha(n)(f^k)(a_1,\dots, a_{j-1}, b, a_{j+1},\dots, a_n)=0, \quad b\in A_0,
\]
if $j\ne k$ ($f\in \Omega $, $n=\nu(f)$, $k=1,\dots, n$).
For $j=k$, one may add
$\alpha(n)(f^q)(a_1,\dots, a_{j-1}, \break b, a_{j+1},\dots, a_n)$
with $q\ne k$
(which is zero) and make sure the result is again in
$A_0$.

Denote $\bar A=A/A_0$. The morphism $\alpha $ induces
a morphism $\bar \alpha : \mathcal P^{(2)} \to \mathcal E(\bar A)$,
such that $(\bar A, \bar \alpha )$ is the quotient
$\mathcal P^{(2)}$-algebra. In this algebra, the values of
algebraic operations $\alpha(n)(f^k)$, $f\in \Omega $, $n=\nu(f)$,
$k\in \{1,\dots, n\}$ do not depend on $k$, so this is actually
an algebra of language $\Omega $:
\[
f^{\bar A}(\bar a_1,\dots, \bar a_n)=\overline{\alpha(n)(f^k)(a_1,\dots, a_n)},
\quad f\in \Omega , \ n=\nu(f),
\]
where $a_j\in A$, $\bar a = a+A_0\in \bar A$.
Moreover, it is obvious that
$\bar A$ is actually a $\mathcal P$-algebra.

Consider the formal direct sum of spaces $\hat A=\bar A\oplus A$
and define algebraic operations of language $\Omega $ on $\hat A$
as follows:
\[
g^{\hat A}(z_1,\dots, z_n) =
 \begin{cases}
    g^{\bar A}(z_1,\dots, z_n),
        & z_i\in \bar A \mbox{ for all }i=1,\dots, n; \\
    \alpha(n)(g^k)(a_1,  \dots,  a_n),&
        z_i=\bar a_i\in \bar A \mbox{ for all }i\ne k, \\
         & z_k=a_k\in A;\\
    0, &\mbox{ more than one }z_i\in A.
 \end{cases}
\]
$g\in \Omega $, $\nu(g)=n$.
Denote by $\hat \alpha $ the corresponding morphism
from $\mathcal F_\Omega $ to $\mathcal E(\hat A)$.

The definition of the canonical section $\rho $ from
Lemma \ref{lem:SurjPerm} and induction on $n$ imply that
for every $s\in \mathcal F_\Omega(n) $,
$s^{\hat A}:= \hat\alpha(n)(s)$, we have
\[
s^{\hat A}(z_1,\dots, z_n) =
\rho(n)(s\otimes e^{(n)}_k)(a_1,  \dots,  a_n)\in A\subseteq \hat A
\]
if $z_i=\bar a_i\in \bar A$  for all $i\ne k$ and
$z_k=a_k\in A$. Therefore, every $s\in S$ is an identity on $\hat A$,
so $(\hat A, \hat\alpha )$ is actually a $\mathcal P$-algebra.

\begin{theorem}\label{thm:EquivDialgDefn}
The kernels of $\pi^{(2)}$ and $\hat\pi^{(2)}$ coincide,
so the operads $\mathcal P^{(2)}$ and $\mathrm{di}\mathcal P$
are equivalent.
\end{theorem}

\begin{proof}
We have already seen that the kernel of $\hat\pi^{(2)}$
is contained in the kernel of $\pi^{(2)}$.
Conversely, assume there exists an identity that holds
on all $\mathrm{di}\mathcal P$-algebras but does not hold
on some $\mathcal P^{(2)}$-algebra $(A,\alpha)$.

Consider the $\mathcal P$-algebra $(\hat A, \hat\alpha)$
constructed above and fix a bialgebra $H$ with a nonzero
$T\in H$ such that $\varepsilon(T)=0$. For example, one may
consider the group algebra $H=\Bbbk\mathbb Z_2$.

The current conformal algebra $\Cur \hat A = H\otimes \hat A$
is a $\mathcal P$-conformal algebra.
By Lemma \ref{lem:ZeroMorphism}, $(\Cur\hat A)^{(0)}$
is a $\mathrm{di}\mathcal P$-algebra. Note that
\[
A \to (\Cur\hat A)^{(0)}, \quad
 a\mapsto 1\otimes \bar a + T\otimes a, \quad a\in A
\]
is an injective homomorphism of $\Omega^{(2)}$-algebras.
Hence, $A$ is in fact $\mathrm{di}\mathcal P$-algebra
and thus satisfies all identities that hold on the class
of such algebras.
\end{proof}

\begin{corollary}\label{cor:EmbConformal}
Every $\mathrm{di}\mathcal P$-algebra $A$
is embedded into $(\Cur \hat A)^{(0)}$
over an appropriate bialgebra $H$, $\hat A$
is a $\mathcal P$-algebra.
\end{corollary}

\begin{lemma}\label{lem:Cur2Alg-Central}
Consider $t=t(x_1,\dots, x_d)\in \mathrm{di}\mathcal P(n)$.
Then $t=t_1\otimes e_1^{(n)}+\dots + t_n\otimes e_n^{(n)}$,
$t_k\in \mathcal P(n)$.
Let $(A,\alpha )$ be a $\mathcal P$-algebra.
Then the dialgebra
$(\Cur A)^{(0)}$ satisfies the identity $t=0$
if and only if $A$ satisfies all identities $t_k=0$, $k=1,\dots, d$.
\end{lemma}

\begin{proof}
It follows from the construction that
if $\alpha(n)(t_k)=0$ for all $k=1,\dots, n$ then
$(\Cur\,\alpha)(n)(t_k)=0$ and hence
$(\Cur\,\alpha)^{(0)}(n)(t)=0$.

Conversely,
consider $g = (\Cur\,\alpha)^{(0)}(n)(t)\in \mathcal E(H\otimes A)(n)$
and compute
\[
b_k= g(1\otimes a_1, \dots , T\otimes a_k, \dots , 1\otimes a_n), \quad
k=1,\dots, n,
\]
for all $a_1,\dots, a_n\in A $.
On the one hand, $b_k=0$ since $(\Cur\,A)^{(0)}$
satisfies the identity $t=0$. On the other hand,
$b_k = T\otimes \alpha(n)(t_k)(a_1,\dots, a_n)$,
so $A$ satisfies $t_k=0$ for all $k$.
\end{proof}

\section{Morphisms of operads and functors}

If $\mathcal P$ and $\mathcal R$ are two operads
then every morphism $\alpha :\mathcal P\to \mathcal R$
gives rise to a functor from the category of
$\mathcal R$-algebras to the category of $\mathcal P$-algebras.
Namely, if $(A,\beta)$ is an $\mathcal R$-algebra
then the same space $A$ with respect to the composition
$\alpha\circ\beta: \mathcal P\to \mathcal E(A)$
is a $\mathcal P$-algebra. The correspondence
$(A,\beta) \mapsto (A, \alpha\circ\beta)$
is obviously functorial. The construction that appears
in Problem \ref{probBSO} is a particular case
of such a functor.

Indeed, assume
$\Omega $ and $\Xi$ are two languages,
$\mathcal F_\Omega $ and $\mathcal F_\Xi$
are two corresponding free operads. Suppose
$\pi:\mathcal F_\Omega \to \mathcal P$
and
$\rho:\mathcal F_\Xi\to \mathcal R$ are
two quotient morphisms to operads
$\mathcal P$ and $\mathcal R$ governing
some varieties of algebras.

Let $\omega :\mathcal F_\Omega \to \mathcal F_\Xi$
be a morphism of operads.
The family of maps
$\omega(n): \mathcal F_\Omega(n) \to \mathcal F_\Xi(n)$
determines (and can be completely determined by) an interpretation
of operations from $\Omega $ via operations from $\Xi $.
We say that $\omega $ induces a morphism
$\bar\omega \mathcal P\to \mathcal R$ iff
$\mathrm{Ker}\,\pi(n)\subseteq
 \mathrm{Ker}\,(\omega (n)\circ \rho(n) )$ for all $n\in \mathbb N$.

\begin{example}
 Let $\Omega $ and $\Xi $ contain one binary operation denoted
by $\lambda $ in $\Omega $ and $\mu $ in $\Xi $, then
the morphism $\omega $ determined by the rule
$\lambda  = \mu  - \mu^{(12)}$, $(12)\in S_2$,
induces a morphism $\bar\omega $
from the operad $\mathrm{Lie}$ (governing the variety
of Lie algebras) to the operad $\mathrm{As}$ (associative algebras).
If $(A,\alpha )$ is an $\mathrm{As}$-algebra
then the pair
$(A, \bar\omega\circ\alpha)$
is exactly the adjoint $\mathrm{Lie}$-algebra of $A$ (usually denoted by $A^{(-)}$).

The Hadamard product $\bar\omega \otimes \mathrm{id}: \mathrm{diLie}\to \mathrm{diAs}$
defines the corresponding functor from the category of associative dialgebras
to the category of Leibniz algebras (Lie dialgebras) \cite{Lod2001}.

A similar relation holds for Mal'cev dialgebras \cite{BPSO} and alternative dialgebras \cite{Liu05}.
\end{example}

\begin{example}
Let $\mathrm{JTS}$ be the operad governing the variety of Jordan triple systems (see, e.g., \cite{Jac49}).
Then there exists $\bar\omega : \mathrm{JTS}\to \mathrm{As}$
defined as follows: If $\tau =(\cdot, \cdot, \cdot)\in \mathrm{JTS}(3)$ is the triple operation on $\mathrm{JTS}$-algebras
and $\mu \in \mathrm{As}(2)$ is the product on associative algebras
then
$\bar \omega (\tau ) = \gamma_{1,2}(\mu;\mathrm{id},\mu) + \gamma_{1,2}(\mu;\mathrm{id},\mu)^{(13)}$.
This is the well-known construction of a Jordan triple system on an associative algebra:
$(a,b,c)=abc+cba$.

In \cite{BSO_Triple}, the notion of a {\em Jordan triple disystem\/} (JTD) was introduced in such a way that
$\mathrm{JTD} = \mathrm{JTS}^{(2)}$, in our notations.
Theorem \ref{thm:EquivDialgDefn} immediately implies $\mathrm{JTD} =\mathrm{diJTS}=\mathrm{JTS}\otimes \Perm$.
Hence, $\bar \omega \otimes \mathrm{id}: \mathrm{JTD} \to \mathrm{diAs}$
defines a structure of a Jordan triple disystem on an associative dialgebra (c.f. \cite[Theorem 5.10]{BSO_Triple}).
\end{example}

\begin{example}
 Let $\mathrm{Jord}$ stand for the operad governing the variety of Jordan algebras,
$\mu \in \mathrm{Jord}(2)$ is the commutative operation.
Then there exists a morphism $\bar\omega : \mathrm{JTS}\to \mathrm{Jord}$ defined by
$\bar \omega (\tau ) = \gamma_{2,1}(\mu;\mu, \mathrm{id}) - \gamma_{2,1}(\mu;\mu, \mathrm{id})^{(23)}
 +\gamma_{1,2}(\mu;\mathrm{id}, \mu) $, i.e.,
$(a,b,c)=(ab)c - (ac)b + a(bc)$.

The notion of a Jordan dialgebra was studied in \cite{VF,Br09,Kol2008}, see also \cite{BPer}.
As in the previous example, Theorem \ref{thm:EquivDialgDefn} gives a new proof
of the relation between Jordan dialgebras and Jordan triple
disystems \cite[Theorem 7.3]{BSO_Triple}.
\end{example}

Let us fix two quotient morphisms
$\pi: \mathcal F_\Omega \to \mathcal P$,
$\rho: \mathcal F_\Xi \to \mathcal R$,
and a morphism
$\omega : \mathcal F_\Omega \to \mathcal F_\Xi$ inducing a
morphism from $\mathcal P$ to $\mathcal R$ which is also denoted by
$\omega$ for simplicity.

\begin{definition}
A $\mathcal P$-algebra $(A,\alpha )$ is called
{\em $\omega $-special\/} if there exists an $\mathcal R$-algebra
$(A,\beta )$ such that $\alpha = \omega \circ \beta $.
The same notion  makes sense for conformal
algebras over an arbitrary cocommutative bialgebra $H$.
\end{definition}

\begin{lemma}\label{lem:CurSpecial}
 If $(A,\alpha )$ is an $\omega $-special $\mathcal P$-algebra
then $(\Cur\, A, \Cur\,\alpha )$
is an $\omega$-special $\mathcal P$-conformal algebra.
\end{lemma}

\begin{proof}
Suppose there exists an $\mathcal R$-algebra $(A,\beta )$,
$\beta : \mathcal R\to \mathcal E(A)$, such that
$\alpha = \omega \circ \beta $.
Then the claim follows from the observation
\begin{equation}\label{eq:OmegaCur}
  \omega \circ \Cur\,\beta = \Cur\,(\omega \circ \beta).
\end{equation}
Indeed, for every $f\in \mathcal P(n)$
the pseudo-linear maps
$(\omega \circ \Cur\,\beta )(n)(f) \in \mathcal E^*(H\otimes A)$
are completely defined by their values at
$(1\otimes a_1,\dots , 1\otimes a_n)$, $a_i\in A$.
By the definition of $\Cur$, we have
$(\Cur\,\beta )(n)(g)(1\otimes a_1,\dots , 1\otimes a_n)
=(1\otimes \dots \otimes 1)\otimes _H (\beta(n)(g))(a_1,\dots , a_n)$
for every $g\in \mathcal R(n)$, in particular, for $g=\omega (n)(f)$.
This is now easy to see that left- and right-hang sides
of \eqref{eq:OmegaCur} coincide at every $f\in \mathcal P(n)$.
\end{proof}

\begin{lemma}\label{lem:ZeroSpecial}
If $(C, \alpha)$ is an $\omega $-special $\mathcal P$-conformal algebra
then $(C, \alpha^{(0)})$ is an $(\omega\otimes \mathrm{id})$-special
$\mathrm{di}\mathcal P$-algebra.
\end{lemma}

\begin{proof}
It is enough to  show
$(\omega\circ\beta)^{(0)} =
(\omega\otimes\mathrm{id})\circ \beta^{(0)}$
for every $\beta: \mathcal R\to \mathcal E^*(C)$.
Relation \eqref{eq:(0)functor} implies
\[
(\omega \circ \beta)^{(0)}(n):
f\otimes e^{(n)}_k \mapsto
(\omega \circ \beta)(n)(f)\circ \mu_n^k
=\beta(n)(\omega (n)(f))\circ \mu_n^k
\]
for every $f\in \mathcal P(n)$, $k=1,\dots, n$.
On the other hand,
\[
((\omega\otimes\mathrm{id})\circ \beta^{(0)})(n):
f\otimes e_k^{(n)} \mapsto
\beta^{(0)}(n) (\omega(n)(f)\otimes e_k^{(n)})
= \beta(n)(\omega (n)(f))\circ \mu_n^k.
\]
\end{proof}

The class of all $\omega $-special $\mathcal P$-algebras is closed
under Cartesian products. Therefore, the class of all homomorphic
images  of all subalgebras of $\omega $-special $\mathcal
P$-algebras  is a variety $\mathfrak S$.
Consider the set of polylinear
identities that hold on this variety and define the corresponding
operad $S^\omega \mathcal P$. This is a quotient operad of
$\mathcal P$, and there exists a morphism $S^\omega : \mathcal P
\to S^\omega \mathcal P$. Every $\mathcal P$-algebra
from $\mathfrak S$ is an $S^\omega \mathcal P$-algebra,
but the converse may not be true if $\mathrm{char}\,\Bbbk>0$.

\begin{lemma}
 Consider the class of all
quotient operads $\mathcal P'$  of $\mathcal P$
satisfying the following property:
For every morphism $\alpha:\mathcal R\to \mathcal E(A)$
there exists a morphism $\alpha': \mathcal P'\to \mathcal E(A)$
such that the diagram
\[
\begin{CD}
\mathcal F_\Omega @>\pi>> \mathcal P @>\pi' >> \mathcal P' \\
@VVV @VV\omega V @VV\alpha' V \\
\mathcal F_\Xi @>\rho >> \mathcal R @>\alpha >> \mathcal E(A)
\end{CD}
\]
 is commutative.
Then $S^\omega \mathcal P$ is a quotient of all such $\mathcal P'$.
\end{lemma}

\begin{proof}
Given $f\in \mathcal P(n)$,
if $\pi'(n)(f)=0$ then $(\omega\circ\alpha)(n)(f)=0$ for every $\alpha $.
For the free countably generated $\mathcal R$-algebra $(A,\alpha)$
each $\alpha(n)$ is injective, so
$\mathrm{Ker}\,\pi'(n) \subseteq \mathrm{Ker}\, \omega(n)$.

By the definition, $S^\omega \mathcal P$ satisfies the condition
on $\mathcal P'$ described above.
Hence,  the kernel of the quotient morphism
$S^\omega : \mathcal P\to S^\omega \mathcal P$
contains all $f\in \mathcal P(n)$ such
that $\omega(n) (f)=0$ in $\mathcal R(n)$.
Therefore,
$\mathrm{Ker}\,\pi'(n) \subseteq \mathrm{Ker}\, \omega(n)\subseteq
\mathrm{Ker}\,S^\omega(n)$.
\end{proof}

\begin{corollary}\label{cor:SpecialKernel}
 The kernel of $S^\omega : \mathcal P\to S^\omega \mathcal P$
coincides with the kernel of $\omega:\mathcal P\to \mathcal R$.
\end{corollary}

\section{Speciality of algebras}
In this section, we state the solution of Problems \ref{probBSO}
and   \ref{probBSOp}.
First, let us reformulate the Problem \ref{probBSO}
in a more general framework.

On the one hand,
$\omega : \mathcal F_\Omega \to \mathcal F_\Xi$
gives rise  to
\[
 \omega \otimes \mathrm{id}: \mathrm{di}\mathcal F_\Omega =\mathcal F_\Omega \otimes \mathrm{Perm}
\to \mathcal F_\Xi \otimes \mathrm{Perm }=\mathrm{di}\mathcal F_\Xi.
\]
Thus, we may define $(\omega\otimes \mathrm{id})$-special
$\mathrm{di}\mathcal P$-algebras
and a variety $\mathfrak S^{(2)}$
generated by them.
Consider the operad $S^{\omega\otimes\mathrm{id}}\mathrm{di}\mathcal P$
 defined by all polylinear identities that hold on
$\mathfrak S^{(2)}$.
This is a generalization of the BSO procedure from \cite{BSO_Triple}.
In the case of zero characteristic, we can conclude that
every $S^{\omega\otimes\mathrm{id}}\mathrm{di}\mathcal P$-algebra
is a homomorphic image of a subalgebra of an
$(\omega \otimes \mathrm{id})$-special $\mathrm{di}\mathcal P$-algebra.

On the other hand, the variety of $S^\omega \mathcal P$-algebras
(defined by polylinear identities)
gives rise to the corresponding variety of
$\mathrm{di}S^\omega\mathcal P$-algebras, where
$\mathrm{di}S^\omega\mathcal P =
S^\omega\mathcal P\otimes \mathrm{Perm}$,
as above.

Both operads
$\mathrm{di}S^\omega\mathcal P$
and
$S^{\omega\otimes\mathrm{id}}\mathrm{di}\mathcal P$
are quotients of
$\mathcal P\otimes \Perm $ and, therefore,
of $\mathcal F_\Omega \otimes \Perm$.

In these terms, Problems \ref{probBSO} and
\ref{probBSOp} are particular cases of the following

\begin{theorem}\label{thm:BSO_general}
(1) If $\mathrm{char}\,\Bbbk =0$ then
$\mathrm{di}S^\omega\mathcal P =
S^{\omega\otimes\mathrm{id}}\mathrm{di}\mathcal P$.

(2) If $\mathrm{char}\,\Bbbk =p>0$ then
the identities of degree $d<p$ that hold on
$\mathrm{di}S^\omega\mathcal P$-algebras are the same
that hold on
$S^{\omega\otimes\mathrm{id}}\mathrm{di}\mathcal P$-algebras
\end{theorem}

\begin{proof}
(1)
We will show that the class of
$\mathrm{di}S^\omega\mathcal P$-algebras coincides with
the class of $S^{\omega\otimes\mathrm{id}}\mathrm{di}\mathcal P$-algebras.

''$\supseteq $'': To compare varieties, we may compare the sets of
defining identities. Suppose
$f=f(x_1,\dots, x_n)\in \mathcal F_{\Omega^{(2)}} (n)$
is a polylinear
identity that holds on the variety
of all $\mathrm{di}S^\omega \mathcal P$-algebras.
Hence,
$\zeta_\Omega (n)(f) = \sum\limits_{k=1}^n f_k\otimes e_k^{(n)}
\in \mathrm{Ker}\,(S^\omega\otimes \mathrm{id})(n)$, so
 $f_k \in \mathrm{Ker}\,S^\omega(n) = \mathrm{Ker}\,\omega(n)$
for all $k=1,\dots, n$, see Corollary \ref{cor:SpecialKernel}.
Therefore, $(\omega \otimes \mathrm{id})(n)(\zeta_\Omega(n)(f)) = 0$,
i.e.,
$\zeta_\Omega(n)(f)$ is an identity of the variety of all
$S^{\omega\otimes\mathrm{id}}\mathrm{di}\mathcal P$-algebras.
Since the kernel of $\zeta_\Omega(n)$ annihilates in $\mathrm{di}\mathcal P$,
we obtain that $f$ is an identity of the variety of all $S^{\omega\otimes\mathrm{id}}\mathrm{di}\mathcal P$-algebras.
Such a relation between identities implies the claim.

Note that this part of the proof does not depend on the characteristic
of the base field.

''$\subseteq $'':
Assume $A$ is a $\mathrm{di}S^\omega\mathcal P $-algebra.
Then by Corollary~\ref{cor:EmbConformal}
$A\subseteq (\Cur\,\hat A)^{(0)}$,
where $\hat A$ is a
$S^\omega\mathcal P $-algebra. Hence,
$\hat A = \varphi(B_1)$, where $\varphi$ is a homomorphism
of $\mathcal P$-algebras, $B_1\subseteq B$, and
$B$ is an $\omega$-special $\mathcal P$-algebra.
Then $\Cur\,B$ is an $\omega$-special conformal $\mathcal P$-algebra
by Lemma~\ref{lem:CurSpecial},
therefore, $(\Cur\,B)^{(0)}$
is an $(\omega\otimes \mathrm{id})$-special
 $\mathrm{di}\mathcal P$-algebra by Lemma~\ref{lem:ZeroSpecial}.
Since $(\Cur\, B_1)^{(0)} \subseteq (\Cur\, B)^{(0)}$
and
$\Cur\,\hat A = \Cur\,\varphi (\Cur\, B_1)$,
we have
$(\Cur\,\hat A)^{(0)} = (\Cur\,\varphi)^{(0)} ((\Cur\,  B_1)^{(0)})$,
i.e., $(\Cur\,\hat A)^{(0)}$ is a homomorphic image of a subalgebra in
an $(\omega\otimes \mathrm{id})$-special
 $\mathrm{di}\mathcal P$-algebra.
Hence, $A$ belongs to the variety of
$S^{\omega\otimes\mathrm{id}}\mathrm{di}\mathcal P$-algebras.

(2)
We have to compare the sets of polylinear identities of degree $d<p$
that hold on all algebras from $\mathfrak S^{(2)}$
and on all $\mathrm{di}S^\omega \mathcal P$-algebras.
Denote the first set by $\mathrm{Id}_d(\mathfrak S^{(2)})$
and the latter one by
$\mathrm{Id}_d(\mathrm{di} S^\omega\mathcal P)$.

The embedding
$\mathrm{Id}_d(\mathrm{di} S^\omega\mathcal P)
\subseteq \mathrm{Id}_d(\mathfrak S^{(2)})$ has already been proved
in part~(1). It remains to prove the converse.

Let $X=\{x_1,x_2, \dots \}$ be a countable set of variables
and $\overline{X}=\{\bar x_1,\bar x_2, \dots \}$ be a copy of $X$.
Consider the free $S^\omega \mathcal P$-algebra
$S^\omega \mathcal P \langle X\cup\overline X\rangle$
generated by $X$ and $\overline{X}$. Since
the class of $ S^\omega\mathcal P$-algebras is a variety
defined by homogeneous identities,
 the notion of degree is well defined for its elements.
Denote by $\deg_X f$ the degree of
$f\in  S^\omega\mathcal P\langle X\cup\overline X\rangle$
with respect to all elements from $X$.

Consider the $ S^\omega\mathcal P$-algebra
\[
F(X) = S^\omega \mathcal P \langle X\cup\overline X\rangle/ J ,
\]
where $J$ is the linear span of all homogeneous
$f\in  S^\omega\mathcal P\langle X\cup\overline X\rangle$
such that $\deg_X f\ge 2$.

Assume
 $t=t(x_1,\dots, x_d) \in \mathrm{Id}_d(\mathcal S^{(2)}) $.
As an element of $\mathrm{di}S^\omega \mathcal P (d)$,
$t$ can be identified with
$t_1\otimes e_1^{(d)}+\dots + t_d\otimes e_d^{(d)}$,
$t_k\in S^\omega \mathcal P(d)$.

There exists a T-ideal $\Sigma $ in
$S^\omega \mathcal P\langle X\cup \overline{X}\rangle $
such that
\[
S^\omega \mathcal P\langle X\cup \overline{X}\rangle /\Sigma \in \mathfrak S.
\]
For every nonzero homogeneous $f\in \Sigma $ we have $\deg f\ge p$
since all identities of lower degree follow from polylinear identities
that hold on $S^\omega \mathcal P$-algebras by definition.
Hence,
\[
F_1(X) = S^\omega \mathcal P \langle X\cup \overline{X}\rangle /(J+\Sigma)\simeq
 F(X)/ ((\Sigma + J)/J)\in \mathfrak S,
\]
so $(\Cur F_1(X))^{(0)}$ satisfies the identity $t=0$.
By Lemma \ref{lem:Cur2Alg-Central},
$F_1(X)$ satisfies all identities $t_k=0$, $k=1,\dots , d$.
But if
$t_k\ne 0$ in $S^\omega \mathcal P(d)$
then
$t_k(\bar x_1,\dots, \bar x_d)\notin \Sigma + J$ in
 $S^\omega \mathcal P\langle X\cup \overline X\rangle $ by the
degree-related reasoning.
Therefore, $t=0$ in $\mathrm{di}S^\omega \mathcal P(d)$.
\end{proof}

Given a triple $(\mathcal P, \mathcal R, \omega )$ as above,
a non-zero $f\in \mathcal P(n)$ is said to be a {\em special identity\/}
if $f\in \mathrm{Ker}\,\omega (n)$.
As a corollary of Theorem \ref{thm:BSO_general},
we may conclude that in all possible settings
$(\mathrm{di}\mathcal P, \mathrm{di}\mathcal R, \omega\otimes\mathrm{id} )$
we should not expect an existence of polylinear special identities different from
$f\otimes e_k^{(n)}$, where $f$ is a
special identity for $(\mathcal P, \mathcal R, \omega )$.
This explains, in particular, the results of \cite{BPSO}, \cite{BPer}, and
\cite{Vor11} concerning special identities
of Jordan and Mal'cev algebras.

To be more precise, state the following

\begin{corollary}
Let $\mathrm{char}\,\Bbbk =0$.
 Suppose $f\in \mathcal F_{\Omega^{(2)}}\langle X \rangle$,
 $X=\{x_1,x_2,\dots \}$, is a polynomial of type $\Omega^{(2)}$.
Denote by $L(f)$ the complete linearization of $f$,
 we may suppose $L(f)\in \mathcal F_{\Omega^{(2)}}(n)$.
Then the following conditions are equivalent.
\begin{itemize}
\item[(S1)]
 $f=0$ is an identity
 on all $(\omega \otimes \mathrm{id})$-special
$\mathrm{di}\mathcal P$-algebras, but
 not an identity on all $\mathrm{di}\mathcal P$-algebras;
\item[(S2)]
$ \zeta_\Omega(n)(L(f)) = \sum\limits_{k=1}^n f_k\otimes e^{(n)}_k$,
 $f_k\in \mathcal F_\Omega$,
where all $f_k$ are identities on the class of all $\omega$-special
$\mathcal P$-algebras and at least one of them is
not an identity on the class of all $\mathcal P$-algebras.
\end{itemize}
\end{corollary}

\begin{proof}
Over a field of characteristic zero, $f$ is a special identity if and only if
so is $L(f)$.
If $L(f)$ does not hold
on all $\mathrm{di}\mathcal P$-algebras
then $(\pi(n)\otimes \mathrm{id})\zeta_\Omega (n)(L(f))\ne 0$, i.e.,
there exists $k$ such that $\pi(n)(f_k)\ne 0$, so $f_k$
does not hold on all $\mathcal P$-algebras. On the other hand,
$(\omega(n)\otimes \mathrm{id})\zeta_\Omega (n) = 0$,
i.e., $\omega (n)(f_k)=0$ for all $k$, so all $f_k$ are identities on
the class of all $\omega$-special $\mathcal P$-algebras.
The proof of the converse statement is similar.
\end{proof}

\end{document}